\documentclass[12pt]{amsart}

\usepackage{amssymb}

\usepackage{eucal}

\usepackage{amsrefs}

\swapnumbers

\theoremstyle{plain}
\newtheorem{theorem}{Theorem}[section]
\newtheorem{proposition}[theorem]{Proposition}
\newtheorem*{cor}{Corollary}
\theoremstyle{remark}
\newtheorem*{rem}{Remark}
\newtheorem*{rems}{Remarks}

\newcommand{\abs}[1]{\left\lvert#1\right\rvert}
\newcommand{\set}[1]{\{\, #1\,\}}
\newcommand{\N}{\mathbb N}
\newcommand{\Q}{\mathbb Q}
\newcommand{\R}{\mathbb R}
\newcommand{\Z}{\mathbb Z}
\newcommand{\C}{\mathbb C}
\newcommand{\rbar}{\overline{\R}}

\makeatletter
   \def\th@plain{\slshape}
   \makeatother

\begin{document}

\title
{Expansions of the real field by canonical products}

\author
{Chris Miller}
\address
{Department of Mathematics\\
Ohio State University\\
231 West~18th Avenue\\
Columbus, Ohio 43210, USA}
\email{miller@math.osu.edu}

\author
{Patrick Speissegger}
\address
{Department of Mathematics \& Statistics\\
McMaster University\\
1280 Main Street West\\
Hamilton, Ontario, Canada L8S 4K1}
\email{speisseg@math.mcmaster.ca}

\thanks{\today. Some version of this document is to appear in the Canadian Mathematical Bulletin.}

\thanks
{Research of Miller partially supported by the Simons Foundation
and by the Mathematisches Forschungsinstitut Oberwolfach (MFO).
Research of Speissegger supported by NSERC of Canada Discovery Grant RGPIN 261961 and the Zukunftskolleg of the University of Konstanz.
Preliminary versions of many of the results herein were announced by Miller at the Fields Institute for Research in Mathematical Sciences
(August 2016) and at MFO (April--May 2017).}

\begin{abstract}
We consider expansions of 
o\nobreakdash-\hspace{0pt}minimal structures on the real field by collections of restrictions to the positive real line of the canonical Weierstrass products associated to sequences such as $(-n^s)_{n>0}$ (for $s>0$) and $(-s^n)_{n>0}$ (for $s>1$), and also expansions by associated functions such as logarithmic derivatives.
There are only three possible outcomes known so far: 
(i)~the expansion is o\nobreakdash-\hspace{0pt}minimal (that is, definable sets have only finitely many connected components); 
(ii)~every Borel subset of each $\R^n$ is definable; 
(iii)~the expansion is interdefinable with a structure of the form $(\mathfrak{R}',\alpha^\mathbb{Z})$ where $\alpha>1$, $\alpha^\mathbb{Z}$ is the set of all integer powers of $\alpha$, and $\mathfrak{R}'$ is o\nobreakdash-\hspace{0pt}minimal and defines no irrational power functions.   
\end{abstract}
 
\keywords{o\nobreakdash-\hspace{0pt}minimal, 
d\nobreakdash-\hspace{0pt}minimal, 
Assouad dimension,
Weierstrass product, 
Gevrey asymptotics}


\maketitle

\section{Introduction}
This work addresses interplay between special functions and model theory (a branch of mathematical logic).
The reader is assumed to be familiar with first-order definability theory over $\rbar:=(\R,+,\cdot,(r)_{r\in\R})$ (the field of real numbers with constants), say, as exposed in van den Dries and Miller~\cite{geocat}.
Throughout, $\Gamma$ indicates the complex Gamma function and $\zeta$ the complex Riemann zeta function.
The set of nonnegative integers is denoted by $\N$.
The symbol $\restriction$ indicates restriction of functions.

A little over twenty years ago, van den Dries and Speissegger established in~\cite{multisummable} that the expansion of $\rbar$ by $\Gamma{\restriction} (0,\infty)$ is o\nobreakdash-\hspace{0pt}minimal (that is, every definable set has only finitely many connected components).
This was accomplished by showing that a larger structure, denoted here by $(\R_\mathcal G,e^x)$, is o\nobreakdash-\hspace{0pt}minimal and defines $\Gamma{\restriction} (0,\infty)$.
Also shown in~\cite{multisummable} is that no restriction of $\zeta$ to any open ray $(c,\infty)$ is definable in $(\R_\mathcal G,e^x)$.
But the bulk of the work in~\cite{multisummable} goes toward obtaining model completeness and o\nobreakdash-\hspace{0pt}minimality of $(\R_\mathcal G,e^x)$, and not much more was done regarding what other special functions are (or are not) definable in $(\R_\mathcal G,e^x)$ or reducts thereof (see~\cite{multisummable}*{\S8});
here, this is one of our main concerns.

The precise definition of the structure $\R_{\mathcal G}$ is lengthy; we shall not repeat it.
(Indeed, we shall need only the definition of the unary primitive functions, but we shall make this more precise later at an appropriate point.)
For present purposes it is enough to know that $\R_\mathcal G$ is o\nobreakdash-\hspace{0pt}minimal, has field of exponents $\Q$
(that is, defines no power functions with irrational exponents)
and expands the better-known structure $\R_\mathrm{an}$ (essentially, the expansion of $\rbar$ by all globally subanalytic sets; see, \textit{e.g.}, \cite{geocat}*{2.5.4}).
For more detailed information on $\R_\mathcal G$ and $(\R_\mathcal G,e^x)$, the best source is still~\cite{multisummable}.  

We begin with a motivating example.
Given $s>0$, let $W_s$ be the canonical Weierstrass product for the sequence $(-n^s)_{n>0}$:
\[
W_s(z)=\prod_{n>0}\left(1+\frac{z}{n^s}\right)\exp
\left(\sum_{j=1}^{\lfloor 1/s\rfloor}\frac{(-1)^jz^j}{jn^{sj}}\right),\ z\in\mathbb C
\]
where $\lfloor\phantom{x}\rfloor$ indicates taking the integer part.
It is classical that $W_s$ is holomorphic, has simple zeros at each $-n^s$ ($n>0$), and no other zeros;
its order (as an entire function) is $1/s$.
As $ze^{\gamma z}W_1=1/\Gamma$ and $\pi z W_2(z^2)=\sinh (\pi z)$, it is reasonable to regard $(s,z)\mapsto W_s(z)$ as a parameterized family of special functions;
indeed, they were investigated as such  in the classical literature, but typically only as individual functions (see, \textit{e.g.}, Barnes~\cite{MR1575933} and Ford~\cite{MR0115035}*{pp.~55--59}) as opposed to a potentially interacting family of functions.
As $W_s$ is real on real, we can also regard $W_s$ as a function $\R\to\R$.
The question arises as to what can be said about expansions of $\rbar$ by collections of restrictions of the $W_s$ to subintervals of $\R$.
As each set $\set{-n^s: n\in\N}$ defines $\Z$ over $\rbar$
(see, \textit{e.g.}, \cite{aph}),
we consider only the case that each subinterval is bounded below.
If $K\subseteq \C$ is compact and subanalytic, then  $W_s{\restriction} K$ is definable in $\R_{\mathrm{an}}$. 
Thus, the restriction of any $W_s$ to any bounded interval is definable in $\R_{\mathrm{an}}$ (hence also in $\R_\mathcal G$).
It is an easy exercise to see that $W_2{\restriction} (0,\infty)$ is interdefinable over $\rbar$ with the function $e^x$.
Hence, by Pfaffian closure~\cite{pfaffcl},
if $\mathfrak{R}$ is an o\nobreakdash-\hspace{0pt}minimal expansion of $\rbar$ then so is the expansion of $\mathfrak{R}$ by $W_2{\restriction} (0,\infty)$.
If $s\in 2\N+4$ (that is, $s$ is an even integer greater than $2$), then any restriction of $W_s$ to any unbounded subinterval of $\R$ defines $\Z$ over $\rbar$ (see Theorem~\ref{mainnegative} below).
As $(\rbar,\Z)$ defines all real projective sets (see, \textit{e.g.}, Kechris~\cite{kechris}*{(37.6)}), we regard expansions of  $(\rbar,\Z)$ as too ``wild'' to be studied as first-order definability theory.   
Hence, we modify our original question:
\emph{What can be said about expansions of $\rbar$ by collections of restrictions $W_s{\restriction} (0,\infty)$ where $s$ ranges over some $S\subseteq (0,\infty)\setminus 2\N+4$?}
Our main working conjecture is that the expansion of $\R_{\mathrm{an}}$ by all $W_s{\restriction} (0,\infty)$ with $s\in (0,\infty)\setminus 2\N+4$ is o\nobreakdash-\hspace{0pt}minimal, but this appears to be beyond our reach at present.
Nevertheless, we do have a number of partial results that illustrate both the potential depth of this conjecture and some important techniques;
we state some of these results now (proofs are mostly deferred to later in the paper).

As is typical when dealing with products, logarithmic derivatives play an important role.
Here, the functions $(zW_s'/W_s)'$ are even more significant.
It is well worth noting that all restrictions of these (meromorphic) functions to compact subintervals of $\R$ are definable in $\R_\mathrm{an}$.

\emph{From now on:} 
$W_s{\restriction} (0,\infty)$ will be denoted by just $W_s$ (any resulting ambiguity should be easily resolved by context).

Our first result is rather easy relative to classical analysis, but it is important.

\begin{proposition}\label{easydefines}
Let $c\geq 0$.
Then $\bigl(\rbar,(xW_s'/W_s)'{\restriction} (c,\infty)\bigr)$ defines the power function $x^s$.
If $1/s\in \N$, then $\bigl(\rbar,(xW_s'/W_s)'{\restriction} (c,\infty)\bigr)$ defines $e^x$.
\end{proposition}

\begin{cor}
$(\rbar,W_1)$ and $(\rbar,\Gamma{\restriction} (0,\infty))$ are interdefinable. 
\end{cor}

\begin{proof}
As $ze^{\gamma z}W_1(z)=1/\Gamma(z)$, the point is to show that $(\rbar, \Gamma{\restriction} (0,\infty))$ defines  $e^x$, which is
immediate from Stirling's Formula: for $x\in\R$, we have 
\[
\lim_{t\to\infty}\frac{\Gamma(x+et)\Gamma(t)}{\Gamma(et)\Gamma(x+t)}=e^x.\qedhere
\]
\end{proof}

Rather more difficult, and one of our main results, is this:

\begin{theorem}\label{main}
If $1<s<2$, then $(\R_{\mathcal G},x^s)$ defines $W_s'/W_s$, and $(\R_\mathcal G, e^x)$ defines $W_s$.
\end{theorem}

By combining these two results with known technology, we obtain:

\begin{theorem}\label{maincor}
$\bigl(\R_\mathcal G,(W_s)_{1\leq s\leq 2}\bigr)$ defines $e^x$,
admits analytic cell decomposition (and so is o\nobreakdash-\hspace{0pt}minimal), and its theory is levelled.\footnote{See Kuhlmann and Kuhlmann~\cite{kk} or Marker and Miller~\cite{levelled}
for the definition of ``levelled theory'' and explanations of its significance.}
If $S\subseteq (1,2)$, then $\bigl(\R_\mathcal G, (W_s'/W_s)_{s\in S}\bigr)$
admits analytic cell decomposition and is polynomially bounded with field of exponents $\Q(S)$.
\end{theorem}
 
These results are sharpened by:

\begin{theorem}\label{mainnegative}
Let $c\geq 0$.
\begin{enumerate}
\item
If $s\in 2\N+4$, then
$\bigl(\rbar,(xW_s'/W_s)'{\restriction} (c,\infty)\bigr)$ defines $\Z$.
\item
If $s>2$, then $\bigl(\R_\mathcal G,(xW_s'/W_s)'{\restriction} (c,\infty)\bigr)$ defines $\Z$.
\item
If $s\neq 2$, then $(\R_\mathrm{an},e^x)$ does not define
$(xW_s'/W_s)'{\restriction} (c,\infty)$.
\end{enumerate}
\end{theorem}

While~(2) might seem to cast doubt on our conjecture,
its proof relies critically on having available certain unary functions that are definable in $\R_\mathcal G$, and these functions are not definable in $(\R_\mathrm{an},e^x)$.
Moreover, it follows routinely from work of Bank~\cite{bank} and Bank and Kaufmann~\cite{bankkauf} that the set of germs at $+\infty$ of the $W_s$ with $s\in 2\N+1$ generates a Hardy field (see any of~\cites{genpower,kk,aph,hardyomin} for a definition).
Hence, if $(\rbar,(W_s)_{s\in 2\N+1})$ is not o\nobreakdash-\hspace{0pt}minimal, then it will be because of something more than just the differential algebra of the ring $\R[x][W_s: s\in 2\N+1]$.

\begin{cor}
Let $S\subseteq (1,2)$ and $\alpha>0$.
If $\alpha>1$, then $\bigl(\R_\mathcal G, (W_s'/W_s)_{s\in S}\bigr)$ defines $W_\alpha'/W_\alpha$ if and only if $\alpha\in \Q(S)\cap (1,2)$.
If $\alpha\in (0,1)\setminus \Q(S)$ or $1/\alpha\in\N$, then $W_\alpha'/W_\alpha$ is not definable in $\bigl(\R_\mathcal G, (W_s'/W_s)_{s\in S}\bigr)$.
\end{cor}
 
\begin{proof}
By Theorem~\ref{maincor}, $\bigl(\R_\mathcal G, (W_s'/W_s)_{s\in S}\bigr)$ is polynomially bounded and has field of exponents $\Q(S)$.
Recall that $W_2$ defines $e^x$ over $\rbar$, and apply Proposition~\ref{easydefines} and Theorem~\ref{mainnegative}.
\end{proof}

As corollaries of proofs, versions of the above results should also hold for certain variants of the $W_s$, but we do not yet properly understand the case $s\in (0,1)$.
Discussion of these issues is best postponed until after the proofs of our results above.

Another family of classical\footnote{See, \textit{e.g.}, the section on Euler partition products in Remmert~\cite{remmert}.} canonical products 
is obtained by interchanging $n$ and $s$ (if $s>1$) in the definition of $W_s$.
For $s>1$, put $F_s=\prod_{n>0}(1+s^{-n}z)$.

\begin{theorem}\label{eulerprods}
Let $s>1$ and $c\in\R$.
\begin{enumerate}
\item
$\bigl(\rbar, F_s'/F_s{\restriction} (c,\infty)\bigr)$
defines $\Z$.
\item
$\bigl(\rbar, (xF_s'/F_s)'{\restriction} (c,\infty)\bigr)$
defines $s^\Z$ (the set of all integer powers of $s$).
\item
If $0<r\notin \Q$, then 
$\bigl(\rbar, (xF_s'/F_s)'{\restriction} (c,\infty), (xF_{s^r}'/F_{s^r})'{\restriction} (c,\infty) \bigr)$
defines $\Z$.
\item
$(\R_\mathrm{an},s^\Z)$ defines
$(xF_{s^q}'/F_{s^q})'{\restriction} (c,\infty)$ for each $0<q\in\Q$.
\end{enumerate}
\end{theorem}

If $s>1$, then $(\R_\mathcal G, s^\Z)$ is clearly not o\nobreakdash-\hspace{0pt}minimal, but it is known to have a number of good qualities;
see~\cite{linear}*{Appendix}, Miller and Thamrongthanyalak~\cite{cpdmin}, and Tychonievich~\cite{tychthesis}*{Theorem~4.1.1}.
By~(4), 
the structure
$\bigl(\R_\mathcal G, ((xF_{s^q}'/F_{s^q})')_{0<q\in \Q}\bigr)$
shares these good qualities.
Thus, in contrast to Theorem~\ref{mainnegative}, the wild behavior of $xF_s'/F_s$ over $\R_\mathcal G$ can be tamed by taking its derivative.
Item~(3) is immediate from~(2) and that if $a,b>0$ and $\{\log a,\log b\}$ is $\Q$\nobreakdash-\hspace{0pt}linearly independent, then $(\rbar,a^\Z,b^\Z)$ defines $\Z$  by Hieronymi~\cite{hier}*{1.3}.

By Theorem~\ref{main}, $\R_\mathcal G$ defines $W_r'/W_r$ if $r\in \Q\cap (1,2)$,
so at least some of the $(xF_s'/F_s)'$ and the $W_r'/W_r$ are mutually well behaved.
However, this is rather exceptional: 

\begin{cor}
Let $\alpha>1$.
If $1/\beta\in\N$ or $\beta\in (0,\infty)\setminus \Q$, then $\Z$ is definable in\newline
$\bigl(\rbar,(xF_\alpha'/F_\alpha)',(xW_\beta'/W_\beta)'\bigr)$.
\end{cor}

\begin{proof}
By Theorem~\ref{eulerprods}, $\bigl(\rbar,(xF_\alpha'/F_\alpha)'\bigr)$ defines $\alpha^\Z$. 
If $1/\beta\in \N$, then
$(\rbar, (xW_\beta'/W_\beta)')$ defines $e^x$ by Proposition~\ref{easydefines}, and $(\rbar,\alpha^\Z,e^x$) evidently defines $\Z$.
If $\beta\in (0,\infty)\setminus \Q$, then $(\rbar, (xW_\beta'/W_\beta)')$ defines $x^\beta$ by Proposition~\ref{easydefines}.
Observe that $(\rbar,\alpha^\Z,x^\beta)$ defines $\alpha^{\beta\Z},$ and again apply 
\cite{hier}*{1.3}.
\end{proof}

We hope we have convinced the reader that there is rich and subtle behavior to be found in the study of expansions of $\rbar$ by canonical products and some of their associated functions.
In this paper, we only scratch the surface. 
Indeed, we should point out that there is a wider context for these investigations, namely, the  ``tameness program'' for expansions of $\rbar$. 
See~\cite{tameness} for the seminal manifesto, but the basic question is: 
\emph{What can be said about expansions of $\rbar$  that do not define $\Z$?} 
In this generality, the most precise answers to date are due to Hieronymi and Miller~\cite{amin}, but it would take us too far afield to discuss details here (though we will employ a special case of the main result of~\cite{amin} in the proof of Theorem~\ref{mainnegative}). 
For the particular setting of this paper, we can make the question a bit more concrete:
 \emph{What can be said about expansions, $\mathfrak{R}$, of $\rbar$ by given collections of analytically interesting functions?}
 We have already seen that there are at least three possibilities:  
 (i)~$\mathfrak{R}$ defines $\Z$ (``as wild as possible''); 
 (ii)~$\mathfrak{R}$ is o\nobreakdash-\hspace{0pt}minimal (``as tame as we could reasonably hope for''); 
 (iii)~$\mathfrak{R}$ is interdefinable with a structure of the form $(\mathfrak{R}', \alpha^\Z)$ where $\alpha>1$ and $\mathfrak{R}'$ is o\nobreakdash-\hspace{0pt}minimal and defines no irrational power functions (``as tame as we could reasonably hope for given that $\mathfrak{R}$ defines some $\alpha^\Z$'').  
 One might wish to conjecture that there are no other possibilities, but this would depend on one's definition of ``analytically interesting''.
To illustrate, it is known (Friedman and Miller~\cite{fast} combined with~\cite{cpdmin}) that for each $p\in \N$  there is a $C^p$ function $f\colon \R\to\R$ having zero set $\set{(n!)!: n\in\N}$ such that $(\rbar,f)$ defines neither $\Z$ nor any $\alpha^\Z$ (but is d\nobreakdash-\hspace{0pt}minimal, as defined in~\cite{tameness}); these functions are constructed by differential calculus methods to have prescribed zero sets, and generally are not of any particular interest otherwise.    
Also known (van den Dries~\cite{densepairs}) is that every locally closed set definable in the expansion of $\rbar$ by the characteristic function of the set of real algebraic numbers is definable in $ \rbar$. 
In any case, it strikes us as reasonable to regard classical special functions as analytically interesting.
 
\begin{rem}
The potential trichotomy has also shown up in investigations of expansions of o\nobreakdash-\hspace{0pt}minimal structures on $\rbar$ 
by locally path-connected trajectories of definable vector 
fields (see~\cites{linear,tychndjfl}).
There is a conjecture that the trichotomy does indeed hold in this setting (even in a stronger form), but there are serious obstacles to progress due to long-standing major open problems in vector field theory and transcendental number 
theory.
We are currently writing a survey of this program.
 \end{rem}
  
\subsubsection*{Acknowledgement}
We are very grateful to Ovidiu and Rodica Costin for many useful communications and much help with the classical analysis.
 
\section{Proofs}
We use the symbol $i$ for the usual complex square root of $-1$.
For nonzero $z\in\C$, we take $\arg z\in (-\pi,\pi]$.
We define differential operators $\mathfrak d_m$ inductively by
$\mathfrak d_0f=f$ and $\mathfrak d_{m+1}f=z\frac{d}{dz}(\mathfrak d_m f)$.
Note that $\mathfrak d\log=1$, so we tend to write $\mathfrak d_m\log f$ instead of $\mathfrak d_m(\log f)$.

We can dispose of Theorem~\ref{eulerprods} without further ado.

\begin{proof}[Proof of Theorem~\ref{eulerprods}]
Let $s>1$ and $c\in\R$.
It suffices to show that:
$\mathfrak d \log F_s{\restriction} (c,\infty)$ defines $\Z$ over $\rbar$;
$\mathfrak d_3 \log F_s{\restriction} (c,\infty)$ defines $s^\Z$ over $\rbar$;
and
$(\R_\mathrm{an},s^\Z)$ defines $\mathfrak d_2 \log F_s{\restriction} (c,\infty)$.
For $m>0$, put $a_m=(-s)^m/(s^m-1)$ and $b_m=\operatorname{csch}(2\pi^2m/\log s)$.
By Littlewood~\cite{littlewood}*{p.~395},\footnote{This work was brought to our attention by O.~Costin.} for $x>1$, we have
\[
\log F_s(x)=\frac{(\log x)^2}{2\log s}-\frac{\log x}{2}+\sum_{m>0}\frac{a_mx^{-m}}{m}-\sum_{m>0}\frac{b_m}{2m}\cos (2\pi m\log_s x).
\]
Thus,
\[
\mathfrak d \log F_s(x)=\log_s x-\frac{1}{2}-\sum_{m>0}a_mx^{-m}
+\frac{\pi}{\log s}\sum_{m>0}b_m\sin(2\pi m\log_s x),
\]
\[
\mathfrak d_2 \log F_s(x)=\frac{1}{\log s}+\sum_{m>0}a_m mx^{-m}
+\frac{2\pi^2}{(\log s)^2}\sum_{m>0}b_mm\cos(2\pi m\log_s x),
\]
and
\[
\mathfrak d_3 \log F_s(x)=-\sum_{m>0}a_m m^2x^{-m}
-\frac{4\pi^3}{(\log s)^3}\sum_{m>0}b_m m^2\sin(2\pi m\log_s x).
\]

If $y\in s^\Z$, then
$ 
\lim_{t\to \infty}\left[(\mathfrak d\log F_s)(yt)
-(\mathfrak d\log F_s)(t)\right]=\log_s y\in\Z;
$
if $y\in (0,\infty)\setminus s^\Z$, then the limit does not exist.
Hence, $\mathfrak d \log F_s{\restriction} (0,\infty)$ defines $\Z$ over $\rbar$.
Only the behavior at $\infty$ of $\mathfrak d\log F_s$ is relevant, so this holds also for $\mathfrak d \log F_s{\restriction} (c,\infty)$.

The zero set of the function
$ 
\sum_{m>0}b_mm^2\sin(\pi m x)
$
is equal to $\Z$, and
\[
-\sum_{m>0}a_mm^2x^{-m}=\frac{s}{s-1}\cdot\frac{1}{x}+o(1/x)\ \text{as $x\to\infty$}.
\]
Hence, the zero set of $\mathfrak d_3 \log F_s\circ x^{1/2}{\restriction} (1,\infty)$ is the image of a sequence $(c_k)_{k>0}$ such that
$\lim_{k\to\infty}(c_{k+1}/c_k)=s$.
By asymptotic extraction of groups (see \cite{aph}*{AEG}),
$\mathfrak d_3\log F_s{\restriction} (1,\infty)$ defines $s^\Z$ over $\rbar$.
Again, only the behavior at $\infty$ of $\mathfrak d_3\log F_s$ is relevant.

The restriction of $ \sum_{m>0}b_m m\cos(2\pi m\log_s x)$ to $[1,s]$ is analytic, hence definable in $\R_\mathrm{an}$.
The function $\lambda\colon (0,\infty)\to  \R$ given by $x\mapsto \max\bigl((0,x]\cap s^\Z\bigr)$ is definable in $(\rbar,s^\Z)$, and
$
\cos(2\pi k\log_s x)=\cos(2\pi k\log_s(\lambda(x)/x))
$
for all $x>0$ and $k\in\Z$.
Thus, $(\R_\mathrm{an},s^\Z)$ defines the function
$
\sum_{m>0}b_m m\cos(2\pi m\log_s x)\colon (0,\infty)\to  \R.
$
As $ \sum_{m>0}a_mmx^{m}$ is analytic on $(-1,1)$, the restriction to $[3/2, \infty)$ of
$
 \sum_{m>0}a_mmx^{-m}
$
is definable in $\R_\mathrm{an}$.
Hence, $(\R_\mathrm{an},s^\Z)$ defines
$\mathfrak d_2\log F_s{\restriction} [3/2, \infty)$.
As $\R_\mathrm{an}$ also defines the restriction of
$\mathfrak d_2\log F_s$ to any bounded interval,
we are done.
\end{proof}

\begin{rems}
(a)~Similarly, if $k>3$, then any restriction of $\mathfrak d_k\log F_s$  to any unbounded-above interval defines $s^\Z$.
Of course, we also have that $(\R_{\mathrm{an}},s^\Z)$ defines $\mathfrak d_k\log F_s$.
(b)~We do not know what to say about $(\R_\mathcal G,F_s{\restriction} (-\infty,0))$ except that it expands $(\R_\mathcal G,s^\Z)$ (observe that  $F_s^{-1}(0)=\set{-s^n: n>0}$).
\end{rems}

We now recall some classical analysis of the functions $W_s$ (say, from~\cite{MR1575933}) that is both motivational and technically useful.
First, suppose that $1/s\notin \N$.
Then there is a contour $C_s$ in the Riemann sphere such that
\begin{multline*}
\log W_s=\pi\csc(\pi/s)x^{1/s}+
\sum_{k=1}^{\lfloor 1/s\rfloor}\frac{(-1)^k}{k}\zeta(sk)x^k
-\frac{\log x}{2}-s\log(2\pi)/2
\\
+\frac{1}{2}\int_{C_s}i\zeta(sw)\csc(\pi w)x^w\,\frac{dw}{w}.
\end{multline*}
The last term has the asymptotic (as $x\to\infty$) expansion
\[
\sum_{k>0}\frac{(-1)^{k+1}}{k}\zeta(-sk)x^{-k}
=
\sum_{k>0}\frac{2(-1)^{k}\sin(k\pi s/2)
\Gamma(1+sk)\zeta(1+sk)}{k(2\pi)^{1+sk}}x^{-k}.
\]
Moreover, for each $m\in\N$,
\begin{multline*}
\mathfrak d_{m}(xW_s'/W_s)=
\frac{\pi}{s^{m+1}}\csc(\pi/s)x^{1/s}+\sum_{k=0}^{\lfloor 1/s\rfloor}(-1)^k\zeta(sk)k^{m}x^{k}\\
+\frac{1}{2}\int_{C_s}i\zeta(sw)\csc(\pi w)w^{m}x^w\, dw,
\end{multline*}
and the last term has the asymptotic expansion
$ \sum_{k>0}(-1)^{k+m}\zeta(-sk)k^{m}x^{-k}$.
(We are deliberately omitting any details about the contour $C_s$.)
Thus, in order to understand the expansion of $\rbar$ by $\mathfrak d_{m}(xW_s'/W_s)$, we must understand the expansion of $(\rbar,x^s)$ by the function
\[
x\mapsto \int_{C_s}i\zeta(sw)\csc(\pi w)w^{m}x^w\,dw,\ x>0.
\]
Unfortunately, this formulation
is of rather limited utility for current purposes, and we shall have to employ other approaches.
Nevertheless, we do acquire some intuition.
Clearly, every term of the formula for $\log W_s$ except the integral is definable in $(\rbar,e^x)$, and every term of the formula for $W_s'/W_s$ except the integral is definable in $(\rbar,x^s)$.
The asymptotic expansions of the integral terms
are Gevrey of order $s$ (or perhaps $1/s$, depending on one's conventions);
this is how $\R_\mathcal G$ came to be under consideration.
But caution is in order.
To illustrate, if $s\in2\N$ then the asymptotic expansion is just $0$, so the only direct conclusion we can draw is that the integral function itself either is eventually $0$ (which is easily seen to be false) or exhibits \emph{some} kind of transpolynomially flat behavior at~$\infty$.
We shall need much more precision.

The case $1/s\in\N$ is similar, but there is an important difference.
As now $\sin (\pi/s)=0$, it no longer makes sense to say that the gross asymptotic behavior of $\log W_s$ is controlled by the term $\pi\csc(\pi/s)x^{1/s}$.
But a limiting argument shows that this can be replaced by $a_sx^{1/s}\log x+b_sx^{1/s}$ for some $a_s$ and $b_s$ with $a_s\neq 0$ (and the degree $\lfloor 1/s\rfloor$ term of the polynomial is omitted).
Everything else is the same.
Both $a_s$ and $b_s$ are explicit, but all we need to know here is that $a_s\neq 0$.

\begin{proof}[Proof of Proposition~\ref{easydefines}]
Put $f=\mathfrak d_2\log W_s$ ($=x(xW_s'/W_s)'$).
We show that $(\rbar,f)$ defines $x^s$, and also $e^x$ if $1/s\in\N$.

Suppose that $1/s\notin \N$.
By~\cite{MR1575933} (as explained in the preceding two paragraphs),
there is a nonzero $c\in\R$ such that $x^{-1/s}f\to c$ as $x\to \infty$.
Then $\lim_{t\to\infty}f(yt)/f(t)=y^{1/s}$ for each $y>0$.
Thus, $x^{1/s}$ is definable (hence also $x^s$).

Suppose that $1/s=m\in \N$.
Again by~\cite{MR1575933}, there is a polynomial $p$ and nonzero $c\in\R$ such that $f=p+cx^m\log x+o(x^m\log x)$.
Put $g=(f-p)/cx^m$.
Then for each $y>0$, we have $\lim_{t\to\infty}[g(yt)-g(t)]=\log y$.
Thus, $\log x$ is definable (hence also $e^x$).
\end{proof}

We now declare two auxiliary functions that will be used in the proofs of Theorems~\ref{main} and~\ref{mainnegative}.
For $s>1$ and $x>0$, we put
\[
\phi_s(x)
=
-2\int_0^\infty
\operatorname{Im}\left[\frac{1}{x+(it)^s}\right]\, \frac{dt}{e^{2\pi t}-1}
\]
and
\[
\omega_s(x)=-\frac{4\pi x^{1/s-1}}{s}\sum_{\substack{0<\theta<\pi/2 \\ \cos(s\theta)=-1}}\operatorname{Im}
\left[\frac{e^{i\theta}}
{\exp(-i2\pi x^{1/s}e^{i\theta})-1}\right].
\]
Note that $\phi_s=0$ if $s\in 2\N+2$,  and $\omega_s=0$ if $1<s\leq 2$ (there are then no $\theta$ as described).
While these definitions of $\phi_s$ and $\omega_s$ are perhaps the easiest to write down, they are not in the most useful forms for current purposes.
Hence, note also that 
\[
\phi_s(x)=2
\sin(s\pi/2)
\sum_{n>0}
\int_0^\infty\frac{t^s e^{-2\pi n t}\, dt}{x^2+2
\cos(s\pi/2)
t^sx+t^{2s}}
\]
and 
\[
\omega_s(x)=\frac{4\pi }{s}x^{1/s-1}\sum_{n>0}
\sum_{\substack{0<\theta<\pi/2 \\ \cos(s\theta)=-1}}\frac{\sin(x^{1/s}2\pi n \cos\theta+\theta)}{\exp(x^{1/s}2\pi n\sin\theta)}.
\]
(In each case, use an appropriate geometric series.) 
 
\begin{proposition}\label{maincalc}
If $s>1$, then
\[
\frac{W_s'}{W_s}=
(\pi/s)\csc(\pi/s)
x^{1/s-1}-\frac{1}{2x}+\phi_s+\omega_s+\begin{cases}
0,& s\notin 4\N+2\\
\dfrac{2\pi x^{1/s-1}}{s(e^{2\pi x^{1/s}}-1)},& s\in 4\N+2
\end{cases}
\]
\end{proposition}

\begin{proof}
Let $s>1$.
First we show that
\[
\frac{W_s'}{W_s}=
(\pi/s)\csc(\pi/s)
x^{1/s-1}-\frac{1}{2x}+2\sum_{n>0}\int_0^\infty\frac{\cos(2\pi n t)\, dt}{x+t^s}.
\]
For each $z\notin \set{-n^s: n>0}$, we have $ 
\frac{W_s'}{W_s}(z)=\sum_{n>0}1/(z+n^s)$.
It is a standard textbook fact that  
$\displaystyle
\int_0^\infty \frac{t^{1/s}\, dt}{t(1+t)}=\pi\csc(\pi/s);
$
a change of variables yields 
$$
\int_0^\infty\frac{dt}{x+t^s}=(\pi/s)\csc(\pi/s)
x^{1/s-1}.
$$
The result now follows by (real-variable) Poisson summation; see, \textit{e.g.}, Titchmarsh~\cite{titchfour}*{\S2.8}.\footnote{But as $s>1$ and $(d/dt)(1/(x+t^s))=-st^{s-1}/(x+t^s)^2$,
it is fair to say that the proof collapses to an exercise in undergraduate analysis.}
 
In what follows, we use the routine fact that if $n\in\N$, $x>0$ and $w\in \C$ is such that $\arg w\neq \pi$ and $x+w^s=0$, then the residue at $w$ of the function
$e^{i2\pi n z}/(x+z^s)$ ($\arg z\neq \pi$) is equal to
$-\frac{1}{s}x^{{1/s}-1}\exp\bigl(i(\arg w+2\pi n w)\bigr)$.

Suppose that $s\notin 4\N+2$.
We show that
$$
2\sum_{n>0}\int_0^\infty\frac{\cos(2\pi n t)\, dt}{x+t^s}=\phi_s+\omega_s.
$$
Let $x>0$.
Let $0<r<R<\infty$ and $C$ be the contour consisting of the union of $[r,R]$, $[ir,iR]$, $\set{re^{i\theta}: 0\leq \theta\leq \pi/2}$,
and $\set{Re^{i\theta}: 0\leq \theta\leq \pi/2}$.
By integrating
$
e^{i2\pi n z}/(x+z^s)
$
over $C$, applying the Residue Theorem, and letting $r\to 0$ and $R\to \infty$,
we obtain
\[
\int_0^\infty\frac{e^{i2\pi nt}}{x+t^s}\, dt=i\int_0^\infty\frac{e^{-2\pi nt}}{x+(it)^s}\, dt -\frac{2\pi ix^{{1/s}-1}}{s}\sum_{\substack{x+w^s=0\\ 0<\arg w<\pi/2}}
\exp\bigl(i(\arg w+2\pi n w)\bigr).
\]
The result now follows by passing to real parts and summing over $n>0$.
 
Suppose that $s\in 4\N+2$.
As $\phi_s=0$, it suffices to show that
\[
2\sum_{n>0}\int_0^\infty\frac{\cos(2\pi n t)\, dt}{x+t^s}=\omega_s+
\frac{2\pi x^{1/s-1}}{s(e^{2\pi x^{1/s}}-1)}.
\]
Because $z\mapsto x+z^s$ is now an even polynomial with no real zeros, we have
\[
2\int_0^\infty\frac{e^{i2\pi n t}\, dt}{x+t^s}=
\int_{-\infty}^\infty \frac{e^{i2\pi n t}\, dt}{x+t^s},
\]
which in turn is equal to $2\pi i$ times the sum of residues of $e^{i2\pi n z}/(x+z^s)$ at the zeros of $x+z^s$ in the upper half plane.
The zeros of $x+z^s$ are symmetric with respect to the imaginary axis and $ix^{1/s}$ is one of them.
Observe that
$$
\sum_{n>0}e^{-2\pi n x^{1/s}}=\frac{1}{e^{2\pi x^{1/s}}-1}
$$
and if $w$ is a zero of $x+z^s$, then so is $-\overline w$, and
$$
\exp \bigl(i(\arg(-\overline w)+2\pi n (-\overline w))\bigr)=-\overline{\exp\bigl(i(\arg w+2\pi n w)\bigr)}.
$$
Recall that $z-\overline{z}=2i\operatorname{Im}z$.
 (We leave the remaining details to the reader.)
\end{proof}

\begin{rem}
For $1<s<2$, the preceding result is almost immediate by the Abel-Plana Formula (see, \textit{e.g.}, \cite{remmert}*{p.~64}) and that $\int_0^\infty 1/(x+t^s)\, dt=(\pi/s)\csc(\pi/s)
x^{1/s-1}$.
But as the function $z\mapsto 1/(x+z^s)$ is not holomorphic at $0$,  it would be more accurate to say ``by a minor variant of Abel-Plana''. 
\end{rem}

\begin{rem}
As $\omega_s$ is an elementary function, so is $W_s'/W_s-\phi_s$.
It follows from standard arguments that if $s>1$, then $W_s$ (as an entire function) is differentially algebraic if and only $\phi_s$ is differentially algebraic.
If $s\in2\N+1$, then by~\cites{bank,bankkauf}, $W_s$ is differentially transcendental (also known as ``hypertranscendental'') over the field of all meromorphic functions $f$ such that
$T(r,f)=o(r^{1/s})$ (where $T(r,f)$ is the Nevanlinna characteristic).
Thus, the collection $\set{\phi_s: s\in 2\N+3}$ is differentially independent over the set of restrictions to the positive real line of entire functions of order $0$.
Among other things, this yields that the germs at $\infty$ of the $\phi_{2n+3}$ all live in a single Hardy field that potentially contains many other interesting analytic germs.
This is a reasonable source of optimism for thinking that the expansion of $\R_\mathrm{an}$ by all of the $W_{2n+1}$ might be o\nobreakdash-\hspace{0pt}minimal.
\end{rem}

The proof of Proposition~\ref{maincalc} required only undergraduate analysis.
This continues through the proof of the first part of the next result, but for the second part, we must assume the reader to understand sections~4.1--4.5 and~5.1  of Balser~\cite{balserbook}.
More precisely, we shall be invoking the conjunction of Proposition~9 and Theorem~22 from~\cite{balserbook}, each of which relies on quite a bit of notation, definitions and conventions.
We see no way to unravel all of this material here for the reader except by rewriting it for our particular needs, which would take us too far afield and add too much length.
For similar reasons, we must assume the reader to have access to~\cite{multisummable}. 

\begin{proposition}\label{intphidef}
If $s>1$, then there exist a differentiable $\Phi_s\colon (0,\infty)\to\R$ and $R>0$ such that $\Phi_s'=\phi_s$ and $\R_\mathcal G$ defines $\Phi_s{\restriction} (R,\infty)$.
\end{proposition}

\begin{proof}
The result is trivial if $s\in2\N$, so assume $s\notin 2\N$.
Let $Q_s(w)$ denote the meromorphic function
\[
\frac{1}{w^2+2\cos(\tfrac{\pi}{2}s)w+1}
=\frac{1}{(w-e^{i\pi(1+s/2)})(w-e^{i\pi(1-s/2)})}.
\]
Let $U_s$ be the complement in $\C$ of the set
$\set{te^{i\pi(1\pm s/2)}: t\geq 1}$.
For $z\in U_s$, the line segment from $0$ to $z$ lies in $U_s$; 
let 
$g_s(z)$ be the path integral of $Q_s$ from $0$ to $z$, that is,   
$ 
g_s(z)=z\int_0^1Q_s(tz)\, dt$.
Note that $g_s(0)=0$, $g_s$ is holomorphic, $g_s'=Q_s$ and
\[
\abs{g_s(z)}
\leq \abs{z}
\max\abs{Q_s{\restriction} [0,z]}=\frac{\abs{z}}{\min_{w\in [0,z]}\bigl(\abs{w-e^{i\pi(1+s/2)}}\abs{w-e^{i\pi(1-s/2)}}\bigr)}.
\]
Let $E\subseteq U_s$ be such that the distance of $E$ to the boundary of $U_s$ is positive; then there exists $C_E>0$ such that $\abs{g_s{\restriction} E}\leq C_E\abs{z}$.
Thus, on $E$, we have
\[
\sum_{n>0}\frac{1}{2\pi n}\abs{g_s(z/(2\pi n)^s)}\leq \sum_{n>0}\frac{C_E\abs{z}}{(2\pi n)^{1+s}}=\frac{C_E\zeta(1+s)}{(2\pi)^{1+s}}\abs{z}.
\]
Define $h_s\colon U_s\to \C$ by
$
h_s(z)=\sum_{n>0}(1/(2\pi n))g_s
(z/(2\pi n)^s)
$;
then $h_s(0)=0$, $h_s$ is holomorphic and $h_s/z$ is bounded on $E$.
For $z$ in the open right half-plane, put
$$
f_s(z)=\int_0^\infty h_s(t^s)e^{-tz}\, dt
=\int_0^\infty h_s(t)e^{-t^{1/s}z} dt^{1/s};
$$
then $f_s$ is holomorphic, and
for all $x>0$,
$$
x^{1/s}f_s(x^{1/s})=\sum_{n>0}\int_0^\infty g_s(t^s/x)e^{-2\pi n t}\, dt.
$$
Thus, $x^{1/s}(f_s\circ x^{1/s})$ is differentiable, and
\begin{align*}
\bigl(x^{1/s}f_s(x^{1/s})\bigr)'
&=\sum_{n>0}\int_0^\infty \frac{d}{dx}g_s(t^s/x)e^{-2\pi n t}\, dt\\
&=
\sum_{n>0}\int_0^\infty \frac{-t^sg_s'(t^s/x)}{x^2}e^{-2\pi n t}\, dt\\
&=-\sum_{n>0}\int_0^\infty \frac{t^sQ_s(t^s/x)}{x^2}e^{-2\pi n t}\, dt\\
&=-\sum_{n>0}
\int_0^\infty\frac{t^s e^{-2\pi n t}}{x^2+2
\cos(s\pi/2)t^sx+t^{2s}}\, dt.
\end{align*}
Put $\Phi_s=-2
\sin(s\pi/2)
x^{1/s}(f_s\circ x^{1/s})$;
then $\Phi_s'=\phi_s$.
It suffices now to find $\epsilon >0$ such that $\R_\mathcal G$ defines the restriction of
$x^{-1/s}(f_s\circ x^{-1/s})$ to $(0,\epsilon)$.
 
Let $\alpha_s>0$ be such that $S(0,\alpha_s,\infty)$ (as defined in~\cite{balserbook}*{4.1}) is the maximal open sector contained in $U_s$ (as before) that contains the positive real axis.
It follows from our earlier computations on the sets $E$ that
$
h_s{\restriction} S(0,\alpha_s,\infty) \in \boldsymbol{A}^{(1/s)}\bigl(S(0,\alpha_s,\infty),\C\bigr)
$.
Let $\mathcal L_{1/s}h_s$ be the Laplace transform of order $1/s$ of $h_s{\restriction}S(0,\alpha_s,\infty)$, defined in a corresponding sectorial region $G=G(0,\alpha_s+\pi s)$.
Recall that $h_s$ is holomorphic at $0$.
By combining~\cite{balserbook}*{Theorem~22} (take $s_1=0$ and $s_2=s$) and~\cite{balserbook}*{Proposition~9},
we obtain that all derivatives of $\mathcal L_{1/s}h_s$ are continuous at the origin, and for every
closed subsector $\overline S$ of $G$ there exist $c,K>0$ such that
\[
\forall n\in\N,\ \frac{1}{n!}\sup_{z\in \overline S}\abs{(\mathcal L_{1/s}h_s)^{(n)}(z)}\leq cK^n\Gamma(1+ns).
\]
Let $N\in\N$ be such that $s<N$ and put $H_s=\mathcal L_{1/s}h\circ z^N$; then again all derivatives of $H_s$
are continuous at the origin and for every
closed subsector $\overline S$ of $\set{z^{1/N}: z\in G}$ there exist $c,K>0$ such that
\[
\forall n\in\N,\ \frac{1}{n!}\sup_{z\in \overline S}\abs{H_s^{(n)}(z)}\leq cK^n\Gamma(1+ns/N).
\]
By Stirling's formula (and adjusting $c$ and $K$), we may replace $\Gamma(1+ns/N)$ with $(n!)^{s/N}$.
Thus, there exist $\varphi\in (\pi/2,\pi)$ and $A,B,\rho>0$ such that for all $n\in\N$ and $z\in \C$, if $0<\abs{z}<\rho$ and $\abs{\arg z}<(s/N)\varphi$, then
$\abs{H_s^{(n)}(z)}/n!\leq AB^n(n!)^{s/N}$.
Choose $b>0$ such that this holds with $\rho>1$ for $H_s\circ bz$.
It follows that
the restriction of $(bz)^{-s/N}f_s\circ (bz)^{-s/N}$ to $(0,1)$ is equal to the restriction to $(0,1)$ of a primitive unary function of $\R_\mathcal G$ (see~\cite{multisummable}*{pp.~514--515}).
Hence, there exists $\epsilon >0$ such that $\R_\mathcal G$ defines
$x^{-1/s}(f_s\circ x^{-1/s}){\restriction}(0,\epsilon)$, as was to be shown.
\end{proof}

The amount of work performed above is not atypical for showing that some given function $(0,\epsilon) \to\R$ is definable in $\R_\mathcal G$ (unless it is globally subanalytic, in which case it is definable in $\R_{\mathrm{an}}$).
Theorem~A of~\cite{multisummable} makes clear that the key is to understand the unary primitive functions of $\R_\mathcal G$, but as we have now seen, this can be difficult.
Another example is from~\cite{multisummable}:
The function $\log \Gamma(x)-(x-1/2)\log x$ ($x>1$)
is definable in $\R_\mathcal G$ (this is used in the proof that $(\R_\mathcal G,e^x)$ defines $\Gamma{\restriction}(0,\infty)$);
the work underlying this assertion is also considerable, but it is not shown in~\cite{multisummable} because it appears in the classical literature 
(\textit{e.g.}, Nielsen~\cite{nielsen}). 
A more modern account can be found in Sauzin~\cite{sauzin}*{\S11}.

\begin{proof}[Proof of Theorem~\ref{main}]
Let $1<s<2$.
We show that $(\R_\mathcal G,x^s)$ defines $W_s'/W_s$ and
$(\R_\mathcal G,e^x)$ defines $W_s$.
Let $\Phi_s$ and $R$ be as in Proposition~\ref{intphidef}.

As $(W_s'/W_s)(z)$ is holomorphic off $\set{-n^s: n>0}$, its restriction to $(0,R]$ is definable in $\R_\mathrm{an}$.
Thus, it suffices to show that $(\R_\mathcal G,x^s)$ defines
$W_s'/W_s{\restriction}(R,\infty)$.
By Proposition~\ref{maincalc},
we have
\[
\frac{W_s'}{W_s}=
(\pi/s)\csc(\pi/s)
x^{1/s-1}-\frac{1}{2x}+\phi_s
\]
($\omega_s=0$ because $1<s<2$).
As $\R_\mathcal G$ defines $\Phi_s{\restriction}(R,\infty)$, it also defines $\phi_s{\restriction}(R,\infty)$ (because $\phi_s=\Phi_s'$).
Hence, $(\R_\mathcal G,x^s)$ defines $W_s'/W_s{\restriction}(R,\infty)$, as was to be shown.

Now we show that $(\R_\mathcal G,e^x)$ defines $W_s$.
By the preceding paragraph,  
\[
\log W_s=
\pi
\csc(\pi/s)
x^{1/s}-\frac{\log x}{2}+\Phi_s+c_s
\]
where $c_s=\log W_s(1)-\pi\csc(\pi/s)+2
\sin(s\pi/2)f_s(1)$. 
Since $\R_\mathcal G$ defines $\Phi_s{\restriction} (R,\infty)$, it is immediate that $(\R_\mathcal G,e^x)$ defines $\log W_s{\restriction} (R,\infty)$, hence also $W_s{\restriction} (R,\infty)$.
Recall that $\R_\mathrm{an}$ defines $W_s{\restriction} (0,R]$.
\end{proof}

\begin{rem}
Put $W(s,z)=W_s(z)$ for $s>0$ and $z\in \C$. 
Given Theorem~\ref{main}, it is natural to wonder about the expansion of $\R_\mathcal{G}$ by the restriction of $W$ to $(1,2)\times (0,\infty)$,
but we do not yet have a satisfactory answer. 
By the proof of Theorem~\ref{main} and some routine definability tricks (recall the proof of Proposition~\ref{easydefines}), 
the expansion of $\rbar$ by $W{\restriction} ((1,2)\times (0,\infty))$ is interdefinable with 
the expansion of 
$
(\rbar, e^x,\arctan x)
$
by the functions $s\mapsto  c_s\colon (1,2)\to\R$  and 
$
(s,x)\mapsto \Phi_s(x) \colon (1,2)\times (0,\infty)\to\R
$.
By~\cite{MR1575933}, it is reasonable to suspect that $c_s=-s\log(2\pi)/2$, but it strikes us as more pressing to understand  $(s,x)\mapsto \Phi_s(x)$, a challenge that we are not yet ready to confront.
\end{rem}

\begin{proof}[Proof of Theorem~\ref{maincor}]
Recall that $W_1$ is interdefinable over $\rbar$ with $\Gamma{\restriction} (0,\infty)$, and 
$W_2$ is interdefinable over $\rbar$ with $e^x$.
Hence, by~\cite{multisummable} and Theorem~\ref{main}, $(\R_\mathcal G,(W_s)_{1\leq s\leq 2})$ is interdefinable with $(\R_\mathcal G,e^x)$, which has analytic cell decomposition (\cite{multisummable}*{9.4}).
By~\cite{kk}, $\operatorname{Th}(\R_\mathcal G,e^x)$ is levelled.
Of course, $(\rbar, e^x)$ expands $(\rbar, (x^r)_{r\in\R})$.
Let $S\subseteq (1,2)$.
By Proposition~\ref{easydefines} and Theorem~\ref{main},
$(\R_\mathcal G, (W_s'/W_s)_{s\in S})$ is interdefinable with $(\R_\mathcal G, (x^s)_{s\in S})$.
By~\cite{multisummable}*{Theorem~A}, $\R_\mathcal G$ has field of exponents $\Q$ (and is o\nobreakdash-\hspace{0pt}minimal).
As $\R_\mathcal G$ expands $\R_\mathrm{an}$, it defines all restricted powers $x^r{\restriction} [1,2]$, $r\in\R$.
By~\cite{linear}*{4.1} (or see~\cite{hardyomin}*{\S5} if a more detailed proof is desired), $(\R_\mathcal G, (x^r)_{r\in \Q(S)})$ has field of exponents $\Q(S)$.
Finally, $(\R_\mathcal G, (x^s)_{s\in S})$ has analytic cell decomposition by~\cite{linear}*{4.1} and essentially the same proof as for $(\R_\mathcal G,e^x)$.
\end{proof}

\begin{rem}
As $\R_\mathcal G$ is model complete (\cite{multisummable}*{Theorem~A}), so is $(\R_\mathcal G, (x^s)_{s\in S})$ (an easy consequence of~\cite{linear}*{4.1}).
\end{rem}

We are ready for the proof of Theorem~\ref{mainnegative}, but
as a convenience to the reader, we first explain a certain technical result that we shall use.
For $\emptyset\neq X\subseteq \R$, we say that $X$ has Assouad dimension~$0$ if for each $\epsilon >0$ there exists $C_\epsilon >0$ such that, for all $0<r<R$ and $x\in X$, the number of intervals of length $2r$ needed to cover $X\cap (x-R,x+R)$ is at most $C_\epsilon (R/r)^\epsilon$.
As a special case of~\cite{amin}*{Theorem~A},
if $E\subseteq \R$ is a finite union of countable locally closed sets, and $f\colon E\to\R$ is continuous and $f(E)$ does not have Assouad dimension~$0$, then $(\rbar, f)$ defines $\Z$.

\begin{proof}
[Proof of Theorem~\ref{mainnegative}]
Let $s>0$ and $c\geq 0$.

Suppose that $2<s\notin 4\N+2$.
We show that $(xW_s'/W_s)'{\restriction} (c,\infty)$ defines $\Z$ over $\R_\mathcal G$.
By Proposition~\ref{easydefines}, $(xW_s'/W_s)'{\restriction} (c,\infty)$ defines $x^{1/s}$ over $\rbar$.
By Propositions~\ref{maincalc} and~\ref{intphidef}, it suffices to show that the zero set of $(x\omega_s)'{\restriction} (c,\infty)$ defines $\Z$ over $\rbar$.
By~\cite{amin}, it suffices to show the zero set of $(x\omega_s)'{\restriction} (c,\infty)$ does not have Assouad dimension~$0$.
The denominators in the double sum appearing in the alternate form of $\omega_s$ are minimized when $n=1$ and $\theta=\pi/s$.
Hence,
\[
\frac{s}{4\pi x^{1/s}}\exp\bigl(x^{1/s}2\pi \sin(\pi/s)\bigr)x\omega_s=
\sin\bigl(x^{1/s}2\pi\cos(\pi/s)+\pi/s
\bigr)+o(1/x^{1/s}).
\]
As there is at least one zero of $(x\omega_s)'$ between each pair of consecutive zeros of
$x\omega_s$, the zero set of $(x\omega_s)'$ is distributed at least as densely as that of
\[
x\mapsto \sin\bigl(x^{1/s}2\pi\cos(\pi/s)+\pi/s\bigr).
\]
The result follows.

Suppose that $s\in 4\N+2$.
We contend that $(xW_s'/W_s)'{\restriction}(c,\infty)$ defines $\Z$ over $\rbar$.
The proof is similar enough to the preceding case that we leave the details to the reader (but recall that $\phi_s=0$ because $s$ is even).

Suppose that $s\neq 2$.
We show that $(\R_\mathrm{an},e^x)$ does not define $(xW_s'/W_s)'{\restriction} (c,\infty)$.
This is clear from the preceding two paragraphs if $s$ is an even integer, so suppose also that $s\notin 2\N$.
By arguing as in the proof of Proposition~\ref{easydefines},
there is a unary function $h$ definable in $(\rbar,e^x)$ such that $(xW_s'/W_s)'-h$ has the asymptotic expansion
$\sum (-1)^{k+1}k\zeta(-sk)x^{-k-1}$.
Thus, it suffices by van den Dries,
Macintyre and Marker~\cite{dmm2}*{5.5}
to show that the formal series
$  \sum k\abs{\zeta(-sk)}T^k$
is divergent.
If otherwise, then
$\limsup_{k\to\infty}\abs{\zeta(-sk)}^{1/k}<\infty$.
Via the Riemann functional equation,
$\lim_{k\to\infty}\abs{\sin(k\pi s/2)}^{1/k}=0$.
But then the subgroup of the unit circle generated by $e^{is\pi/2}$ has no elements with nonzero imaginary part, and so $s/2\in\N$, contradicting that $s\notin 2\N$.
\end{proof}

\section*{Concluding remarks}

The functions $W_s$ can be regarded as special cases of canonical products for more general sequences with terms such as $-(a+bn_1+\dots b_Nn_N)^s$ for suitably chosen real numbers $a,b_1,\dots,b_N$. 
Indeed, yet more general forms are treated in considerable detail in~\cites{MR1575933}.
We imagine that results similar to ours could hold for many of these functions (especially if $N=1$) but the computational details could be daunting.

We do not yet understand the $W_s$ for $0<s<1$ as well as we would like, especially if $1/s\in\N$ (but recall Proposition~\ref{easydefines}).
By arguments similar to those in the proof of Proposition~\ref{maincalc}, the problem reduces to understanding 
$$
x\mapsto \sum_{n>0}\int_1^\infty\frac{\cos(2\pi n t)\, dt}{t^{s\lfloor 1/s\rfloor}(x+t^s)}
$$ 
for large $x$, but we have not yet been able to do so.
We do have reason to suspect that there is a polynomially bounded o\nobreakdash-\hspace{0pt}minimal expansion $\mathfrak{R}$ of $\R_\mathcal G$ such that
the conclusion of Theorem~\ref{main} should hold with $\mathfrak{R}$ in place of $\R_\mathcal G$ if $-1<\cos(\pi/s)<0$, and the conclusion of
Theorem~\ref{mainnegative} should hold with $\mathfrak{R}$ in place of $\R_\mathcal G$ if 
$0< \cos(\pi/s)<1$.

An important difference between the $W_s$ and $F_s$ (as in Theorem~\ref{eulerprods}) is that every $W_s$ has finite positive order (as an entire function) while every $F_s$ has order $0$.
We expect that order~$0$ products will generally be poorly behaved.
Indeed, if $(a_n)_{n\geq 1}$ is a sequence of positive real numbers such that $\liminf_{n\to\infty}a_{n+1}/a_n\geq 100$, then $\prod(1+x/a_n)$ defines $\Z$ over $\rbar$.\footnote{This is a special case of a result due to O.~Costin and author Miller that arose as part of an ongoing project on Hardy fields generated by canonical products. 
The value $100$ is chosen only for computational convenience and proof of concept; it is probably too large.}
We sketch the proof.
Put $W(z)=\prod(1+z/a_n)$ and let $c\geq 0$.
We show that $W'/W{\restriction} (c,\infty)$ defines $\Z$ over $\rbar$.
Put $f=xW'/W$ and let $Z$ be the zero set of $\mathfrak{d}_3\log W{\restriction} (c,\infty)$.
Observe that $f$ is continuous and $Z$ is discrete.
It suffices now by~\cite{amin} to show that $f(Z)$ does not have Assouad dimension~$0$.
If $0<k\in\N$, then $k/2<\sum_{n\leq k}1/(1+a_n/a_k)<k(1+a_1/a_k)$.
Thus, for each $\epsilon \in (0,1/2)$ and sufficiently large $k$ (depending on $\epsilon$), $f$ maps the interval $(a_k,a_{k+1})$ into the interval $((1/2-\epsilon)k, (1+\epsilon)k)$.
As $f'>0$, the images $f\bigl((a_k,a_{k+1})\bigr)$ are pairwise disjoint.
It is a calculus exercise that
for all sufficiently large $k$ (depending only on $W$)
there is an element of $Z$ in $(a_k,a_{k+1})$.\footnote{While not easy, it can be given to undergraduate students with the hints to sketch the graph of $\mathfrak{d}_3\log (1+x/a)$ for fixed $a>0$,
then use that $\mathfrak{d}_3\log W$ is the sum of the $\mathfrak{d}_3\log (1+x/a_n)$ and that the $a_n$ are growing very rapidly.}
Thus, given $\epsilon\in (0, 1/2)$,
for all sufficiently large $k$ (depending on $\epsilon$) there exists $b_k\in f(Z)$ such that $(1/2-\epsilon)k<b_k<(1+\epsilon)k$.
It follows that $f(Z)$ does not have Assouad dimension~$0$ (as was to be shown).
Observe that for $s\geq 100$, we have proven that $F_s'/F_s$ defines $\Z$ over $\rbar$ (recall Theorem~\ref{eulerprods}) without using complex analysis.
On the other hand, as no $\mathfrak d_3 \log F_s$ defines $\Z$ even over $\R_\mathcal G$,
it can happen that $\mathfrak d_3 \log W$ does not define $\Z$ over $\rbar$ by itself.

One can also consider real-on-real products that might have infinitely many nonreal zeros, but this tends to go beyond the methods of this paper.
Indeed, prompted by the introduction, we ask: 
What can be said about the expansion of $\R_\mathcal G$ by $\zeta{\restriction} (1,\infty)$?
We conjecture that it is o\nobreakdash-\hspace{0pt}minimal.  
Something that is known: The expansion of $\R_\mathrm{an}$ by $\zeta{\restriction} (1,\infty)$ is o\nobreakdash-\hspace{0pt}minimal; indeed, by \cite{multisummable}, this holds over $\R_{\mathrm{an}^*}$ (see~\cite{genpower} for the definition), another polynomially bounded o\nobreakdash-\hspace{0pt}minimal expansion of $\R_\mathrm{an}$.
Ongoing work of Speissegger with J.-P.~Rolin and T.~Servi is aimed at proving that $(\R_\mathcal  G,\zeta{\restriction} (1,\infty))$ is o\nobreakdash-\hspace{0pt}minimal by establishing that $\R_\mathcal G$ and $\R_{\mathrm{an}*}$ have a common polynomially bounded o\nobreakdash-\hspace{0pt}minimal expansion.
It is worth noting  that $(\rbar,\zeta{\restriction} (1,\infty))$ defines $e^x$ (consider $x\mapsto \lim_{t\to\infty}(\zeta(t)-1)/(\zeta(x+t)-1)$), and so $(\R_\mathcal  G,\zeta{\restriction} (1,\infty))$ 
defines $\Gamma{\restriction}(0,\infty)$.
But as mentioned in the introduction, $\zeta{\restriction} (1,\infty)$ is not definable in $(\R_\mathcal  G, e^x)$, hence also not in $(\R_\mathcal  G, \Gamma{\restriction}(0,\infty))$.
Thus, relative to $\R_\mathcal  G$,  $\zeta{\restriction} (1,\infty)$ carries more information than $\Gamma{\restriction}(0,\infty)$. 

We close with a brief return to our main conjecture and the tameness program. 
Forgetting Theorems~\ref{main} through~\ref{mainnegative} for the moment,
let $S\subseteq (0,\infty)$.
Can it be that $(\R_\mathrm{an},(W_s)_{s\in S})$ neither is o\nobreakdash-\hspace{0pt}minimal nor defines $\Z$?
If o\nobreakdash-\hspace{0pt}minimality fails, then
by~\cite{opencore} and~\cite{hier2}*{Lemma~2} there is a definable infinite discrete $E\subseteq (0,\infty)$ such that if $x,y\in E$ and $x\neq y$, then $\abs{x-y}\geq 1$.
If definability of $\Z$ fails, then $E$ has Assouad dimension~$0$, as does the image of $E$ under each finite compositional iterate of the compositional inverse $W_s^{-1}$ of $W_s$ for each $s\in S$.
As $W_s^{-1}$ grows roughly like $(\log x)^{s}$, this suggests that $E$ should be somehow ``transexponentially sparse''.
(We should point out that we have used here only a rather special case of the main result of~\cite{amin}.)
This seems unlikely, particularly if $S\subseteq 2\N+1$, as then the set of germs of the $W_{2n+1}$ generate a Hardy field (by results from~\cites{bank,bankkauf}).
More generally, let $F$ be a nonpolynomial entire function and $M_F$ denote its maximum-modulus function.
(Note that if $s>1$ then $M_{W_s}=W_s{\restriction} [0,\infty)$ and $M_{F_s}=F_s{\restriction} [0,\infty)$.)
What can be said about $(\R_\mathrm{an}, M_F)$?
As every restriction of $F$ to a compact disc is definable in $\R_\mathrm{an}$, so is the restriction of $M_F$ to any compact subinterval of $[0,\infty)$.
Thus, once again, the point is to understand the behavior of $M_F$ at $\infty$.
Recall that $M_F$ is strictly increasing and transpolynomial (by the Maximum Principle and Liouville's Theorem).
By arguing similarly as before, it seems likely that $(\R_\mathrm{an}, M_F)$ either is o\nobreakdash-\hspace{0pt}minimal or defines $\Z$ (and any other possibilities should be rather esoteric).

\bibsection

\end{document}